\documentclass{amsart}
\usepackage{hyperref}
\usepackage{enumerate}
\usepackage[alphabetic]{amsrefs}

\theoremstyle{plain}
\newtheorem{thm}{Theorem}[section]
\newtheorem{defn}[thm]{Definition}

\newtheorem{rem}[thm]{Remark}
\newtheorem{prop}[thm]{Proposition}
\newtheorem{lem}[thm]{Lemma}
\newtheorem*{thm*}{Theorem}
\newtheorem*{rem*}{Remark}

\title{Hopf coactions on odd spheres}
\author{Suvrajit Bhattacharjee}
\address{Stat-Math Unit\\ Indian Statistical Institute\\ 203, B.T. Road\\ Kolkata-700108}
\email{suvra.bh\_r@isical.ac.in}
\author{Debashish Goswami}
\address{Stat-Math Unit\\ Indian Statistical Institute\\ 203, B.T. Road\\ Kolkata-700108}
\email{goswamid@isical.ac.in}

\begin{document}
	
\begin{abstract}
We prove that the q-deformed unitary group, i.e., $U_q(N)$, is the universal compact quantum group in the category of (compact) quantum groups which coact on the q-deformed odd sphere $S_q^{2N-1}$ leaving the space spanned by the natural set of generators invariant and preserving the unique $SU_q(N)$ invariant functional on $S_q^{2N-1}$. Using this, we identify $U_q(N)$ as the quantum group of orientation preserving isometries (in the sense of Bhowmick and Goswami \cite{MR2555012}) for a natural spectral triple associated with $S_q^{2N-1}$ constructed by Chakraborty and Pal \cite{MR2458039}. 
\end{abstract}

\maketitle

\section{Introduction}

Symmetry is one of the most fundamental and ubiquitous concepts in mathematics and other areas of science. Classically, symmetry of some mathematical structure is understood as the group of automorphisms or isomorphisms (in a suitable sense) of the structure. Generalizing the concept of group, Drinfel\cprime d and Jimbo (\cites{MR934283, MR797001, MR841713}) introduced the notion of quantum groups. Later on, Woronowicz (\cites{MR901157, MR1616348}) formulated an analogue of this notion in the analytical framework of $C^*$-algebras.\\

It is natural to conceive of quantum groups as `symmetry objects' for certain mathematical structures, e.g., rings or $C^*$-algebras. Indeed, Manin (\cite{MR1016381}) pioneered the idea of viewing q-deformation of classical Lie groups (e.g. $SL_q(N)$) as some kind of `quantum automorphism group'.\\

A similar approach was taken by Wang (\cite{MR1637425}) who defined quantum permutation group of finite sets and quantum automorphism group of finite dimensional matrix algebras. This was followed by flurry of work by several other mathematicians including Banica, Bichon and others (\cites{MR2146039, MR2174219, MR1937403}). The second author of the present article and his collaborators (including Bhowmick, Skalski and others) approached the problem from a geometric perspective and formulated an analogue of the Riemannian isometry groups in the framework of (compact) quantum groups acting on $C^*$-algebras. We refer the reader to \cite{MR3559897} and the references therein for a comprehensive account of the theory of quantum isometry groups. See also \cite{MR3130334,MR3609348}.\\

A useful procedure of producing genuine examples of `non-commutative spaces' is to deform the coordinate algebra or some other suitable function algebra underlying a classical space. In this context, it is natural to ask the following: what is the quantum isometry group of a non-commutative space obtained by deforming a classical space? It is expected that under mild assumptions, it should be isomorphic with a deformation of the isometry group of the classical space, at least when the classical space is connected. Indeed, for a quite general class of cocycle deformation (called the Rieffel deformation), such a result has been proved by Bhowmick, Goswami and Joardar (\cite{MR3261868}). See also \cites{MR3582020, MR3611308}.\\

However, no such general result has yet been achieved for the Drinfel\cprime d-Jimbo type q-deformation of semisimple  Lie groups and the corresponding homogeneous spaces. The goal of the present paper is to make some progress in this direction. We have been able to prove the above result for q-deformed odd spheres, i.e., $S_q^{2N-1}$ (\cite{MR1086447}). Classically (for $q=1$), these are nothing but the complex $N-1$ dimensional spheres $\{(z_1,\cdots,z_N) \mid \sum_i |z_i|^2=1 \}$ inside $\mathbb{C}^N$. The universal group that can act `linearly', i.e., leaves the span of the complex coordinates $z_1,\cdots, z_N$ invariant and also preserves the canonical inner-product on $span\{z_1,\cdots,z_n\}$ coming from the standard inner-product of $\mathbb{C}^N$, is the unitary group $U(N)$.\\

We have proved in Section 6 (Theorem \ref{6.6}) a q-analogue of this result. More precisely, we have proved the following theorem.

\begin{thm*}
Let Q be a Hopf $\ast$-algebra coacting on $\mathcal{O}(S_q^{2N-1})$ by $\rho$ making it a $\ast$-comodule algebra, where we have viewed $\mathcal{O}(S_q^{2N-1})$  as a $\ast$-coideal subalgebra of $\mathcal{O}(SU_q(N))$. Moreover, suppose that
\begin{enumerate}[i)]
\item $\rho$ leaves the subspace $V=span\{z_1,\cdots, z_N\}$ invariant, i.e., $\rho(z_i)=\sum_j z_j \otimes q^j_i$ for some $q^i_j \in Q$. We write $\textbf{q}=(q^i_j)$ for the matrix of $Q$-valued coefficients;
\item $\rho$ preserves the inner product on $V$ induced by the Haar functional.
\end{enumerate} Then there is a unique $\ast$-morphism $\Psi : \mathcal{O}(U_q(N)) \rightarrow Q$ such that $(id \otimes \Psi)\rho_u=\rho$. 
\end{thm*} 

Using this, we have also identified (Theorem \ref{7.12}) $U_q(N)$ with the (orientation preserving) quantum isometry group of a natural spectral triple on $S_q^{2N-1}$ constructed in \cite{MR2458039}.

\begin{rem*}
We can compare the above result with \cite{MR3609348}*{Theorem 1.1}. In \cite{MR3609348}, the algebra $\mathcal{A}$ on which Hopf coactions are considered is commutative, which is replaced by a q-deformed quantized function algebra in the present article. Moreover, flexibility of choice of a non-degenerate bilinear form in \cite{MR3609348}*{Theorem 1.1, Condition (i)} is gone in our case; we have the somewhat rigid requirement of preserving a canonical non-degenerate sesquilinear form coming from the Haar functional. In fact, a special advantage of working with a commutative algebra in \cite{MR3609348} is that any bilinear form on $\mathcal{A}$ admits a natural extension (as a bilinear form) on $\mathcal{A} \otimes \mathcal{A}$ which is invariant under the flip map. This is no longer true for quantized function algebra, if we consider the obvious q-analogue of the flip, associated with the natural braiding.  
\end{rem*}

\section{Coquasitriangular Hopf Algebras}

In this and the following sections we introduce some well known material on cosemisimple Hopf algebras and coquasitriangular Hopf algebras. The main object of investigation, the Vaksman-Soibelman (also called quantum or odd) sphere, is also introduced.\\

Let $H$ be a Hopf algebra with comultiplication $\Delta$, counit $\epsilon$, antipode $S$, unit $1$ and multiplication $m$. We use Sweedler notation throughout, i.e., for the coproduct we write $\Delta(a)=a_1 \otimes a_2$ and for a right coaction $\rho$, we write $\rho(x)=x_0 \otimes x_1$. 

\begin{defn} \label{2.1} \cite{MR1492989}*{page 331}
A coquasitriangular Hopf algebra is a Hopf algebra $H$ equipped with a linear form $\textbf{r} : H \otimes H \rightarrow \mathbb{C}$ such that the following conditions hold:
	
\begin{enumerate}[i)]
\item $\textbf{r}$ is invertible with respect to the convolution, that is, there exists another linear form $\overline{\textbf{r}} : H \otimes H \rightarrow \mathbb{C}$ such that $\textbf{r}\overline{\textbf{r}}=\overline{\textbf{r}}\textbf{r}=\epsilon \otimes \epsilon$ on $H \otimes H$;
\item $m_{H^{op}}=\textbf{r} \ast m_H \ast \overline{\textbf{r}}$ on $H \otimes H$;
\item $\textbf{r}(m_H \otimes id)=\textbf{r}_{13}\textbf{r}_{23}$  and $\textbf{r}(id \otimes m_H)=\textbf{r}_{13}\textbf{r}_{12}$ on $H \otimes H \otimes H,$
\end{enumerate}
	
where $\textbf{r}_{12}(a \otimes b \otimes c)=\textbf{r}(a \otimes b)\epsilon(c)$, $\textbf{r}_{23}(a \otimes b \otimes c)=\epsilon(a)\textbf{r}(b \otimes c)$ and $\textbf{r}_{13}(a \otimes b \otimes c)=\epsilon(b)\textbf{r}(a \otimes c)$, $a,b,c$ in $H$.
\end{defn}

\begin{rem} \label{2.2}
A linear form $\textbf{r}$ on $H \otimes H$ with the properties $(i)-(iii)$ is called a universal r-form on $H$.
\end{rem}

Since linear forms on $H \otimes H$ correspond to bilinear forms on $H \times H$, we can consider any linear form $\textbf{r} : H \otimes H \rightarrow \mathbb{C}$ as a bilinear form on $H \times H$ and write $\textbf{r}(a,b):=\textbf{r}(a \otimes b)$, $a,b \in H$. Then the above conditions $(i)-(iii)$ read as 
\begin{equation} \label{eq1}
\textbf{r}(a_1,b_1)\overline{\textbf{r}}(a_2,b_2)=\overline{\textbf{r}}(a_1,b_1)\textbf{r}(a_2,b_2)=\epsilon(a)\epsilon(b),
\end{equation}
\begin{equation} \label{eq2}
ba=\textbf{r}(a_1,b_1)a_2b_2\overline{\textbf{r}}(a_3,b_3),
\end{equation}
\begin{equation} \label{eq3}
\textbf{r}(ab,c)=\textbf{r}(a,c_1)\textbf{r}(b,c_2),
\end{equation}
\begin{equation} \label{eq4}
\textbf{r}(a,bc)=\textbf{r}(a_1,c)\textbf{r}(a_2,b),
\end{equation} 
with $a,b,c \in H$.

\begin{rem} \label{2.3} ~\cite{MR1492989}*{page 334}
It can be shown that $\textbf{r}(S(a),S(b))=\textbf{r}(a,b)$.
\end{rem}

Let $H$ be a coquasitriangular Hopf algebra with universal r-form $\textbf{r}$. For right $H$-comodules $V$ and $W$ we define a linear mapping $\textbf{r}_{V,W} : V \otimes W \rightarrow W \otimes V$ by 
\begin{equation} \label{eq5}
\textbf{r}_{V,W}(v \otimes w)=\textbf{r}(v_1,w_1) w_0 \otimes v_0,
\end{equation} $v \in V$, $w \in W$.

\begin{rem} \label{2.4} ~\cite{MR1492989}*{page 333}
It can be shown that $\textbf{r}_{V,W}$ is an isomorphism of the right $H$-comodules $V \otimes W$ and $W \otimes V$.
\end{rem}

The compatibility of a universal r-form and a $\ast$ structure is described in the following definition.

\begin{defn} \label{2.5} ~\cite{MR1492989}*{page 336}
A universal r-form $\textbf{r}$ of a Hopf $\ast$-algebra $H$ is called real if $\textbf{r}(a \otimes b)=\overline{\textbf{r}(b^* \otimes a^*)}$.
\end{defn}

\section{The quantum semigroup $M_q(N)$}

Let $q$ be a positive real number. We now introduce some of the well known deformations of classical objects.

\begin{defn} \label{3.1} ~\citelist{\cite{MR1492989}*{page 310}\cite{MR1015339}}
The FRT bialgebra, also called the coordinate algebra of the quantum matrix space, denoted $\mathcal{O}(M_q(N))$, is the free unital $\mathbb{C}$-algebra with a set of $N^2$ generators $\{u^i_j \mid i,j=1,\cdots, N\}$ and defining relations 
\begin{equation}\label{eq6}
u^i_ku^j_k=qu^j_ku^i_k, \qquad u^k_iu^k_j=qu^k_ju^k_i, \qquad i < j,
\end{equation}
\begin{equation}\label{eq7}
u^i_lu^j_k=u^j_ku^i_l, \qquad i < j, \qquad k < l,
\end{equation}
\begin{equation}\label{eq8}
u^i_ku^j_l - u^j_lu^i_k=(q-q^{-1})u^j_ku^i_l, \qquad i < j, \qquad k < l.
\end{equation}
\end{defn}

\begin{prop}\label{3.2}
There is a unique bialgebra structure on the algebra $\mathcal{O}(M_q(N))$ such that 
\begin{equation}\label{eq9}
\Delta(u^i_j)=\sum_k u^i_k \otimes u^k_j, \qquad and \qquad \epsilon(u^i_j)=\delta_{ij}, \qquad i,j=1,\cdots,N.
\end{equation}
\end{prop}

The above construction can be realized more conceptually as follows. Let $\hat{R} : \mathbb{C}^N \otimes \mathbb{C}^N \rightarrow \mathbb{C}^N \otimes \mathbb{C}^N$ be the linear operator whose matrix with respect to the standard basis of $\mathbb{C}^N$ is given by
\begin{equation}\label{eq10}
\hat{R}^{ij}_{mn}=q^{\delta_{ij}}\delta_{in}\delta_{jm} + (q-q^{-1})\delta_{im}\delta_{jn}\theta(j-i),
\end{equation}
where $\theta$ is the Heaviside symbol, that is, $\theta(k)=1$ if $k > 0$ and $\theta(k)=0$ if $k \leq 0$. Let the inverse of $\hat{R}$ be $\hat{R}^-$. Also, let $\check{R}$ be the ``dual" operator defined by $\check{R}^{ij}_{mn}=\hat{R}^{nm}_{ji}$ and $\check{R}^{-ij}_{mn}=\hat{R}^{-ij}_{mn}$.

\begin{rem} \label{3.3} \cite{MR1492989}*{page 309}
It is known that $\hat{R}$ satisfies 
\begin{equation}\label{eq11}
(\hat{R} - qI)(\hat{R} + q^{-1}I)=0,
\end{equation} 
where $I$ is the identity operator.
\end{rem}

The following shows that $M_q(N)$ is universal in a sense.

\begin{prop} \label{3.4} \cite{MR1492989}*{page 305}
\begin{enumerate}[i)]
\item There is a linear map $\phi : \mathbb{C}^N \rightarrow \mathbb{C}^N \otimes \mathcal{O}(M_q(N))$ such that $\mathbb{C}^N$ is a right comodule of $\mathcal{O}(M_q(N))$ with coaction $\phi$ and $\hat{R}$ is a comodule morphism, i.e., $(\hat{R} \otimes id)\phi^{(2)}=\phi^{(2)}\hat{R}$, where $\phi^{(2)}$ is the induced coaction on $\mathbb{C}^N \otimes \mathbb{C}^N$ given by $\phi^{(2)}(v \otimes w)=v_0 \otimes w_0 \otimes v_1w_1$ ($\phi(v)=v_0 \otimes v_1$);
\item If $A$ is any other bialgebra and $\psi : \mathbb{C}^N \rightarrow \mathbb{C}^N \otimes A$ is a right coaction of $A$ on $\mathbb{C}^N$ such that $\hat{R}$ is a comodule morphism (in the sense described above) then there exists a unique bialgebra morphism $\Theta : \mathcal{O}(M_q(N)) \rightarrow A$ such that $(id \otimes \Theta)\phi=\psi$. 
\end{enumerate} 
\end{prop}

Since $\mathcal{O}(M_q(N))$ is only a bialgebra, we want to construct a Hopf algebra out of it. For that we need the following definition.

\begin{defn}\label{3.5} \cite{MR1492989}*{page 312}
The quantum determinant, denoted $\mathcal{D}_q$, is the element of $\mathcal{O}(M_q(N))$ defined by
\begin{equation}\label{eq12}
\sum_{\pi \in S_N} (-q)^{\ell(\pi)}u^1_{\pi(1)}\cdots u^N_{\pi(N)},
\end{equation}
where $S_N$ is the symmetric group on $N$ letters and $\ell(\pi)$ is the number of inversions in $\pi.$
\end{defn}

\begin{rem}\label{3.6}\cite{MR1492989}*{page 312, 313}
It is an important fact that $\mathcal{D}_q$ is central, nonzero and group-like in $\mathcal{O}(M_q(N))$. We recall that group-like means $\Delta(\mathcal{D}_q)=\mathcal{D}_q \otimes \mathcal{D}_q$. Applying $(\epsilon \otimes id)$ on the identity $\Delta(\mathcal{D}_q)=\mathcal{D}_q \otimes \mathcal{D}_q$ yields $\mathcal{D}_q=\epsilon(\mathcal{D}_q)\mathcal{D}_q$, hence $\epsilon(\mathcal{D}_q)=1$ as $\mathcal{D}_q\neq 0$.
\end{rem}

\section{The quantum group $SU_q(N)$}

The deformation of the special linear group is realized as follows.

\begin{defn}\label{4.1} \cite{MR1492989}*{page 314}
The coordinate algebra of the quantum special linear group is defined to be the quotient
\[\mathcal{O}(SL_q(N))=\mathcal{O}(M_q(N))/\langle \mathcal{D}_q - 1 \rangle\]
of the algebra $\mathcal{O}(M_q(N))$ by the two-sided ideal generated by the element $\mathcal{D}_q - 1.$
\end{defn}

The following shows that $\mathcal{O}(SL_q(N))$ is indeed a Hopf algebra.

\begin{prop}\label{4.2}
There is a unique Hopf algebra structure on the algebra $\mathcal{O}(SL_q(N))$ with comultiplication $\Delta$ and counit $\epsilon$ such that 
\begin{equation}\label{eq13}
\Delta(u^i_j)=\sum_k u^i_k \otimes u^k_j \qquad and \qquad \epsilon(u^i_j)=\delta_{ij}.
\end{equation}
The antipode $S$ of the Hopf algebra is given by
\begin{equation}\label{eq14}
S(u^i_j)=(-q)^{i-j}\sum_{\pi \in S_{N-1}}(-q)^{\ell(\pi)}u^{k_1}_{\pi(l_1)}\cdots u^{k_{N-1}}_{\pi(l_{N-1})},
\end{equation} 
where $\{k_1,\cdots, k_{N-1}\}:=\{1,\cdots,N\}\setminus\{j\}$ and $\{l_1,\cdots,l_{N-1}\}:=\{1,\cdots,N\}\setminus\{i\}$ as ordered sets.
\end{prop}

The composite $\mathbb{C}^N \xrightarrow{\phi} \mathbb{C}^N \otimes \mathcal{O}(M_q(N)) \rightarrow \mathbb{C}^N \otimes \mathcal{O}(SL_q(N))$ gives the natural coaction of $\mathcal{O}(SL_q(N))$ on $\mathbb{C}^N$.\\

As quantum groups are understood to be quasitriangular Hopf algebras, quantum function algebras are assumed to be coquasitriangular Hopf algebras. We want to think of $\mathcal{O}(SL_q(N))$  as the quantum function algebra of $SL(N)$.

\begin{thm}\label{4.3} \cite{MR1492989}*{page 339}
$\mathcal{O}(SL_q(N))$ is a coquasitriangular Hopf algebra with universal r-form $\textbf{r}_t$ uniquely determined by
\begin{equation}\label{15}
\textbf{r}_t(u^i_j \otimes u^k_l)=t\hat{R}^{ki}_{jl},
\end{equation}
where $t$ is the unique positive real number such that $t^N=q^{-1}$.
\end{thm}

\begin{rem} \label{4.4}
It can be shown that the morphism $\textbf{r}_{\mathbb{C}^N,\mathbb{C}^N}$ induced by the universal r-form $\textbf{r}_t$ of $\mathcal{O}(SL_q(N))$, equals $t\hat{R}$. Let us denote it by $\sigma.$
\end{rem}

The following resembles complex conjugation.

\begin{prop}\label{4.5} \cite{MR1492989}*{page 316}
There is a unique $\ast$-structure on the Hopf algebra $\mathcal{O}(SL_q(N))$ given by $(u^i_j)^*=S(u^j_i)$, making it into a Hopf $\ast$-algebra.
\end{prop}

Let us now introduce the quantum version of the real form $SU(N)$ of $SL(N)$.

\begin{defn}\label{4.6}
The coordinate algebra of the quantum special unitary group $SU_q(N)$ is the Hopf $\ast$-algebra $\mathcal{O}(SL_q(N))$ of the above proposition.
\end{defn}

\begin{thm}\label{4.7} \cite{MR1492989}*{page 340}
The universal r-form $\textbf{r}_t$ of $\mathcal{O}(SU_q(N))$ is real, in the sense of Definition \ref{2.5}.
\end{thm}

The algebraic counterpart of classical Peter-Weyl theorem for compact groups is contained in the following:

\begin{defn}\label{4.8} \cite{MR1492989}*{page 403}
A Hopf algebra $H$ is said to be cosemisimple if there exists a unique linear form $\textbf{h} : H \rightarrow \mathbb{C}$, which we call the Haar functional, such that $\textbf{h}(1)=1$, and for all $a$ in $H$
\[(id \otimes \textbf{h}) \Delta(a)=\textbf{h}(a)1,
\qquad
(\textbf{h} \otimes id) \Delta(a)=\textbf{h}(a)1.\]
\end{defn}

At this point, it is natural to recall the analytic counterpart of compactness.

\begin{defn}\label{4.9} \cite{MR1616348}
A compact quantum group (CQG) is given by a pair $(S,\Delta)$, where $S$ is a unital $C^*$ algebra and $\Delta$ is a unital $C^*$ homomorphism $\Delta : S \rightarrow S \otimes S$ ($\otimes$ is $C^*$-algebraic minimal tensor product) satisfying
\begin{enumerate}[i)]
\item $(\Delta \otimes id) \circ \Delta=(id \otimes \Delta) \circ \Delta;$
\item Each of the linear spans of $\Delta(S)(S \otimes 1)$ and $\Delta(S)(1 \otimes S)$ is norm-dense in $S \otimes S$.
\end{enumerate}
\end{defn}

If the comultiplication is understood, we simply denote the compact quantum group by $S$. It is well known (\cite{MR1616348}) that there is a canonical dense $\ast$-algebra $\mathcal{S}$ of $S$, consisting of the matrix coefficients of the finite dimensional unitary representation (to be defined in Section 7) of $S$, which becomes a Hopf $\ast$-algebra with comultiplication $\Delta|_{\mathcal{S}}$.\\

We say that the compact quantum group $(S,\Delta)$ acts on a unital $C^*$-algebra $B$, if there is a unital $C^*$-homomorphism $\alpha : B \rightarrow B \otimes S$ satisfying
\begin{enumerate}[i)]
\item $(\alpha \otimes id) \circ \alpha=(id \otimes \Delta) \circ \alpha;$
\item the linear span of $\alpha(B)(1 \otimes S)$ is norm-dense in $B \otimes S$.
\end{enumerate}

Given such an action, there exists (\cite{MR1637425}) a dense unital $\ast$-subalgebra $\mathcal{B}$ of $B$ such that $\mathcal{B}$ is a $\ast$-comodule algebra over $\mathcal{S}$.

\begin{defn}\label{4.10} \cite{MR1492989}*{page 416}
A compact quantum group (CQG) algebra is a cosemisimple Hopf $\ast$-algebra $H$ with Haar functional $\textbf{h}$ such that $\textbf{h}(a^*a)>0$, for all $a \neq 0.$
\end{defn}

It is known (\cite{MR1310296}) that $H$ is a $CQG$ algebra if and only if it is isomorphic to the dense Hopf $\ast$-algebra $\mathcal{S}$ of a compact quantum group $S$. The following captures the compactness of the real form $SU(N)$.

\begin{thm}\label{4.11} \cite{MR1492989}*{page 418}
$\mathcal{O}(SU_q(N))$ is a CQG algebra.
\end{thm}

In fact, there is a natural $C^*$-norm on $\mathcal{O}(SU_q(N))$. Upon completion with respect to this norm, one gets the unital $C^*$-algebra $SU_q(N)$ which is a compact quantum group in the sense of Definition \ref{4.9}. Moreover, $SU_q(N)$ is the universal $C^*$-algebra generated by $\mathcal{O}(SU_q(N))$.\\

We end this section by recalling a standard fact. Let $V$ be a finite dimensional vector space with inner-product $\langle, \rangle$. Let $H$ be a Hopf $\ast$-algebra with a coaction $\rho : V \to V \otimes H$ on $V$. We say that $\rho$ preserves the inner-product if $\langle v_0,w_0\rangle v_1^*w_1=\langle v,w\rangle 1$ for all $v,w \in V$, where we use Sweedler notation, i.e., $\rho(v)=v_0 \otimes v_1$ and likewise for $\rho(w)$.

\begin{prop}\label{4.12}\cite{MR1492989}*{page 402}
If $\rho$ preserves the inner-product on $V$ and $W$ is a subcomodule of $V$ then the orthogonal complement $W^{\perp}$ (with respect to $\langle, \rangle$) is also a subcomodule of $V$.
\end{prop}


\section{The odd sphere}

We now introduce the main example to be studied. We remark that these are $q$-deformations of the complex $(N-1)$-dimensional (hence odd dimensional real) spheres. One could as well define $q$-deformations of even dimensional real spheres which are quantum homogeneous spaces of $\mathcal{O}(SO_q(n))$ for appropriate $n$. We do not deal with these because our proof of some of the main results crucially use the fact that $\hat{R}$ has two eigenvalues whereas in the case of $\mathcal{O}(SO_q(n))$, it has three.

\begin{defn}\label{5.1}\citelist{\cite{MR1682006}*{page 2} \cite{MR1086447}}
The coordinate algebra $\mathcal{O}(S_q^{2N-1})$ of the quantum sphere is the free unital $\mathbb{C}$-algebra with a set of \hspace{.09cm}2N generators $\{z_i,z^*_i \mid i=1,\cdots,N\}$ and defining relations
\begin{equation}\label{eq16}
z_iz_j=qz_jz_i, \qquad z^*_iz^*_j=q^{-1}z^*_jz^*_i, \qquad i < j
\end{equation}
\begin{equation}\label{eq17}
z_iz^*_j=qz^*_jz_i, \qquad i \neq j
\end{equation}
\begin{equation}\label{eq18}
z_iz^*_i-z^*_iz_i+q^{-1}(q-q^{-1})\sum_{k > i}z_kz^*_k=0,
\end{equation}
\begin{equation}\label{eq19}
\sum_{i=1}^{N}z_iz^*_i=1,
\end{equation}
together with the $\ast$-structure $(z_i)^*=z^*_i$ and $(z^*_i)^*=z_i.$
\end{defn}

There is a natural $C^*$-norm on $\mathcal{O}(S_q^{2N-1})$. Upon completion with respect to this norm, one gets the unital $C^*$-algebra $S_q^{2N-1}$ which is the universal $C^*$-algebra generated by $\mathcal{O}(S_q^{2N-1})$.

\begin{prop}\label{5.2}\cite{MR2458039}*{page 34}
Putting $z_i=u^1_i$ and $z^*_i=(u^1_i)^*=S(u^i_1)$ gives an embedding of $\mathcal{O}(S_q^{2N-1})$ into $\mathcal{O}(SU_q(N))$ making $\mathcal{O}(S_q^{2N-1})$ into a quantum homogeneous space for $\mathcal{O}(SU_q(N))$ with the coaction
\begin{equation}\label{eq20}
\Delta_R(z_i)=\sum_j z_j \otimes u^j_i, \qquad \Delta_R(z^*_i)=\sum_j z^*_j \otimes S(u^i_j).
\end{equation}
\end{prop}

The compact quantum group $SU_q(N)$ acts on $S_q^{2N-1}$ in the $C^*$-algebraic sense, lifting the above coaction on $\mathcal{O}(S_q^{2N-1})$. Classically, the sphere $S^{2N-1}$ can be described as a homogeneous space for $SU(N)$. The above proposition states the quantum version of it. Moreover, one can view the sphere as a homogeneous space for $U(N)$ also. As expected, the last statement continues to hold in the quantum world too. We need the following definition.

\begin{defn}\label{5.3}\cite{MR1492989}*{page 313}
The coordinate algebra of the quantum general linear group is defined to be the quotient
\[\mathcal{O}(GL_q(N))=\mathcal{O}(M_q(N))[t]/\langle t\mathcal{D}_q - 1 \rangle\]
of the polynomial algebra $\mathcal{O}(M_q(N))[t]$ in $t$ over $\mathcal{O}(M_q(N))$  by the two-sided ideal generated by the element $t\mathcal{D}_q - 1.$
\end{defn}

\begin{prop}\label{5.4}
There is a unique Hopf algebra structure on the algebra $\mathcal{O}(GL_q(N))$ with comultiplication $\Delta$ and counit $\epsilon$ such that 
\begin{equation}\label{eq21}
\Delta(u^i_j)=\sum_k u^i_k \otimes u^k_j \qquad and \qquad \epsilon(u^i_j)=\delta_{ij}.
\end{equation}
The antipode $S$ of the Hopf algebra is given by	\begin{equation}\label{eq22}	S(u^i_j)=(-q)^{i-j}{\mathcal{D}_q}^{-1}\sum_{\pi \in S_{N-1}}(-q)^{\ell(\pi)}u^{k_1}_{\pi(l_1)}\cdots u^{k_{N-1}}_{\pi(l_{N-1})} \quad \text{and} \quad S(\mathcal{D}_q)={\mathcal{D}_q}^{-1},
\end{equation} 	where $\{k_1,\cdots, k_{N-1}\}:=\{1,\cdots,N\}\setminus\{j\}$ and $\{l_1,\cdots,l_{N-1}\}:=\{1,\cdots,N\}\setminus\{i\}$ as ordered sets.
\end{prop}

We have the following analogue of the real form $U(N)$.

\begin{prop}\label{5.5} \cite{MR1492989}*{page 316}
There is a unique $\ast$-structure on the Hopf algebra $\mathcal{O}(GL_q(N))$ given by $(u^i_j)^*=S(u^j_i)$, making it into a Hopf $\ast$-algebra.
\end{prop}  

\begin{defn}\label{5.6}
The coordinate algebra of the quantum unitary group $U_q(N)$ is the Hopf $\ast$-algebra $\mathcal{O}(GL_q(N))$ of the above proposition. In this $\ast$-algebra, the quantum determinant $\mathcal{D}_q$ becomes a unitary element.
\end{defn}

In analogy with $U(N)$, the fact that $U_q(N)$ is a compact quantum group is reflected in the following theorem.

\begin{thm}\label{5.7} \cite{MR1492989}*{page 418}
$\mathcal{O}(U_q(N))$ is a CQG algebra.
\end{thm}

Thus we have the analogues of Proposition \ref{4.5} and Theorem \ref{4.10}. Moreover, Proposition \ref{5.2} remains true in this case and makes $\mathcal{O}(S_q^{2N-1})$ a quantum homogeneous space for $\mathcal{O}(U_q(N))$. There is a natural $C^*$-norm on $\mathcal{O}(U_q(N))$. Upon completion with respect to this norm, one gets the unital $C^*$-algebra $U_q(N)$ which is a compact quantum group in the sense of Definition \ref{4.9}. $U_q(N)$ is also the universal $C^*$-algebra generated by $\mathcal{O}(U_q(N))$ and it acts on $S_q^{2N-1}$ in the $C^*$-algebraic sense lifting the algebraic coaction on $\mathcal{O}(S_q^{2N-1})$. We end this section with a lemma.

\begin{lem}\label{5.8}
Suppose $Q$ is a Hopf $\ast$-algebra and $q^i_j \in Q$, $i,j=1,\cdots,N$ such that the following hold:
\begin{enumerate}[i)]
\item $\Delta(q^i_j)=\sum_k q^i_k \otimes q^k_j$ and $\epsilon(q^i_j)=\delta_{ij}$;
\item $q^i_j$ satisfy the FRT relations \eqref{eq6}, \eqref{eq7} and \eqref{eq8};
\item $\textbf{q}$ satisfies $\textbf{q}\textbf{q}^*=I_n$, where $\textbf{q}=(q^i_j)$ and $\textbf{q}^*=\overline{\textbf{q}}^t$, $\overline{\textbf{q}}=((q^i_j)^*)$.
\end{enumerate}
Then there is a unique $\ast$-morphism $\Psi : \mathcal{O}(U_q(N)) \to Q$ between these Hopf $\ast$-algebras such that $\Psi(u^i_j)=q^i_j$. 
\end{lem}

\begin{proof}
Define $\Phi : \mathcal{O}(M_q(N)) \to Q$ on the generators $u^i_j$ by $\Phi(u^i_j)=q^i_j$. Then by hypotheses i), ii), Definition \ref{3.1} and Proposition \ref{3.2}, $\Phi$ extends to a bialgebra morphism. Using Remark \ref{3.6}, we conclude:
\begin{equation}\label{23}
\Delta(\Phi(\mathcal{D}_q))=(\Phi \otimes \Phi)\Delta(\mathcal{D}_q)=\Phi(\mathcal{D}_q) \otimes \Phi(\mathcal{D}_q)
\end{equation}
and
\begin{equation}\label{24}
\epsilon(\Phi(\mathcal{D}_q))=\epsilon(\mathcal{D}_q)=1.
\end{equation} Thus \eqref{23}, \eqref{24} imply that $\Phi(\mathcal{D}_q)$ is a nonzero group-like element in $Q$. Moreover, applying $(S \otimes id)$ and $(id \otimes S)$ on either side of \eqref{23}, we observe $S(\Phi(\mathcal{D}_q))\Phi(\mathcal{D}_q)=\Phi(\mathcal{D}_q)S(\Phi(\mathcal{D}_q))=\epsilon(\Phi(\mathcal{D}_q))1=1$. Hence $\Phi(\mathcal{D}_q)$ is invertible with $\Phi(\mathcal{D}_q)^{-1}=S(\Phi(\mathcal{D}_q))$ We also have $q^i_j\Phi(\mathcal{D}_q)=\Phi(\mathcal{D}_q)q^i_j$, by the centrality of $\mathcal{D}_q$, (see Remark \ref{3.6}), hence $q^i_j$ commute with $\Phi(\mathcal{D}_q)^{-1}=S(\Phi(\mathcal{D}_q))$.\\

Now define $\tilde{\Phi} : \mathcal{O}(M_q(N))[t] \to Q$ by \[\tilde{\Phi}(u^i_j)=q^i_j, \quad \text{and} \quad \tilde{\Phi}(t)=S(\Phi(\mathcal{D}_q)).\] This is well-defined because $q^i_j$ commute with $S(\Phi(\mathcal{D}_q))$. Clearly, $\tilde{\Phi}$ sends the ideal generated by $t\mathcal{D}_q-1$ to $0$. Hence $\tilde{\Phi}$ descends to $\Psi : \mathcal{O}(GL_q(N)) \to Q$ which is, by Proposition \ref{5.4}, a bialgebra morphism.\\

We are left to show that $\Psi$ is a $\ast$-map, with respect to the Hopf $\ast$-algebra structure on $\mathcal{O}(GL_q(N))$ defined in Proposition \ref{5.5} i.e., the Hopf $\ast$-structure of $\mathcal{O}(U_q(N))$. To this end, applying $m(S \otimes id)$ and $m(id \otimes S)$ on either side of the relation $\Delta(q^i_j)=\sum_k q^i_k \otimes q^k_j$ (hypothesis i)), it follows that the antipode $S$ satisfies \[S(\textbf{q})\textbf{q}=\textbf{q}S(\textbf{q})=I_N,\] i.e., $S(\textbf{q})=\textbf{q}^{-1}$, where $S(\textbf{q})=(S(q^i_j))$. This, together with hypothesis iii), implies $S(\textbf{q})=\textbf{q}^*$ i.e., $S(q^i_j)=(q^j_i)^*$. Since $\Psi$ is a bialgebra morphism, it preserves the antipode and so Proposition \ref{5.5} implies that $\Psi$ is a $\ast$-morphism. Finally, the proof of uniqueness is straightforward, hence omitted.
\end{proof}

\section{Main results}

We start with a basic observation.

\begin{lem}\label{6.1}
Let $H$ be any cosemisimple Hopf $\ast$-algebra with the Haar functional $\textbf{h}$. Then for $a,b \in H$, $\textbf{h}(a^*_1b)a^*_2=\textbf{h}(a^*b_1)S(b_2).$ 
\end{lem}

\begin{proof}
By definition, $\textbf{h}(x)1=\textbf{h}(x_1)x_2$ for $x \in H$. Applying the antipode $S$, we get $\textbf{h}(x)1=\textbf{h}(x_1)S(x_2)$. Now,
\[
\begin{aligned}
\textbf{h}(a^*_1b)a^*_2&= \textbf{h}(a^*_1b_1)S(a^*_2b_2)a^*_3\\
&=\textbf{h}(a^*_1b_1)S(b_2)S(a^*_2)a^*_3\\
&=\textbf{h}(a^*_1b_1)S(b_2)\epsilon(a^*_2)\\
&=\textbf{h}(a^*_1\epsilon(a^*_2)b_1)S(b_2)\\
&=\textbf{h}(a^*b_1)S(b_2).
\end{aligned}
\]
\end{proof}

The following lemma exploits the relation between the two apparently different concepts, namely coquasitriangularity and faithfulness of the Haar functional.

\begin{lem}\label{6.2}
Let $H$ be a coquasitriangular CQG algebra with Haar functional $\textbf{h}$ and real universal r-form $\textbf{r}$. Let the induced inner-product be denoted by $ \langle,\rangle$, i.e., $\langle a,b \rangle=\textbf{h}(a^*b)$. Let $V$ be any subcomodule of $H$ and $r_{V,V}$ be the induced morphism on $V \otimes V$. Then $r_{V,V}$ is hermitian with respect to the restricted inner-product.
\end{lem}

\begin{proof}
We have
\[
\begin{aligned}
\langle r_{V,V}(v \otimes w), v' \otimes w' \rangle
&=\overline{\textbf{r}(v_1,w_1)} \langle w_0 \otimes v_0, v' \otimes w' \rangle\\
&=\textbf{r}(w^*_1,v^*_1) \langle w_0,v' \rangle \langle v_0,w' \rangle \qquad (\text{we use reality of \textbf{r}})\\
&=\textbf{r}(\textbf{h}(w^*_0v')w^*_1, \textbf{h}(v^*_0w')v^*_1)\\
&=\textbf{r}(\textbf{h}(w^*v'_0)S(v'_1),\textbf{h}(v^*w'_0)S(w'_1)) \qquad (\text{using Lemma \ref{6.1}})\\
&=\textbf{r}(S(v'_1),S(w'_1)) \langle w,v'_0 \rangle \langle v,w'_0 \rangle\\
&=\textbf{r}(v'_1,w'_1) \langle v \otimes w, w'_0 \otimes v'_0 \rangle \qquad (\text{by Remark \ref{2.3}})\\
&=\langle v \otimes w, r_{V,V}(v' \otimes w') \rangle.
\end{aligned}
\]
\end{proof}

\begin{lem}\label{6.3}
Let $A$ and $Q$ be Hopf $\ast$-algebras, $B \subset A$ a $\ast$-coideal subalgebra of $A$ and a comodule algebra over $Q$ with coaction $\rho : B \rightarrow B \otimes Q$ such that $\rho(b^*)=\rho(b)^*$ for all $b \in B$. Suppose that $A$ is cosemisimple and $\rho$ preserves the restriction of the Haar functional $\textbf{h}$ of $A$ to $B$ i.e., $(\textbf{h} \otimes id)\rho(b)=\textbf{h}(b)1$ for all $b \in B$. Then $\rho$ preserves the induced inner-product on $B$ given by $\langle a, b \rangle=\textbf{h}(a^*b)$, in the sense described in the paragraph preceding Proposition \ref{4.12}.
\end{lem}

The proof is straightforward, hence omitted.


\begin{prop}\label{6.4}
Let $\pi : \mathcal{O}(U_q(N)) \rightarrow \mathcal{O}(SU_q(N))$ be the quotient homomorphism and $\rho_{u}$, $\rho_{su}$, respectively, be the corresponding coactions on $\mathcal{O}(S_q^{2N-1})$ so that $(id \otimes \pi)\rho_{u}=\rho_{su}$. Then $\rho_u$ preserves the restriction of the Haar functional $\textbf{h}$ on $\mathcal{O}(S_q^{2N-1})$.
\end{prop}

\begin{proof}
Let $f$ be any linear functional on $\mathcal{O}(S_q^{2N-1})$ such that $f(1)=1$. Let $f'$ be defined as $(f \otimes h_u)\rho_u$, where $h_u$ is the Haar functional of $U_q(N)$. Then $f'$ is $\rho_u$ invariant. By $(id \otimes \pi)\rho_{u}=\rho_{su}$, we get that $f'$ is $\rho_{su}$ invariant too. It is well known that the restriction of $\textbf{h}$ is the only functional with this property \cite{MR1492989}. Hence, the conclusion follows.
\end{proof}

We think the following is well known. We included it because we couldn't find it in the literature.

\begin{prop}\label{6.5}
The set $\{z^{k_1}_1\cdots z^{k_N}_N(z^*_{N-1})^{l_{N-1}}\cdots (z^*_1)^{l_1}, z^{k_1}_1\cdots z^{k_{N-1}}_{N-1}(z^*_N)^{l_N}\cdots (z^*_1)^{l_1} \mid k_1,\cdots,k_N,l_{N-1},\cdots, l_1 \in \mathbb{N}\cup\{0\}, \quad l_N \in \mathbb{N}\}$ is a vector space basis of $\mathcal{O}(S_q^{2N-1}).$
\end{prop}

\begin{proof}
It is a simple application of Bergman's Diamond lemma ~\cite{MR506890}.
\end{proof}

We state and prove below the main result concerning Hopf coactions on the quantum spheres satisfying suitable conditions.

\begin{thm}\label{6.6}
Let Q be a Hopf $\ast$-algebra coacting on $\mathcal{O}(S_q^{2N-1})$ by $\rho$ making it a $\ast$-comodule algebra, where we have viewed $\mathcal{O}(S_q^{2N-1})$  as a $\ast$-coideal subalgebra of $\mathcal{O}(SU_q(N))$. Moreover, suppose that
\begin{enumerate}[i)]
\item $\rho$ leaves the subspace $V=span\{z_1,\cdots, z_N\}$ invariant, i.e., $\rho(z_i)=\sum_j z_j \otimes q^j_i$ for some $q^i_j \in Q$. We write $\textbf{q}=(q^i_j)$ for the matrix of $Q$-valued coefficients;
\item $\rho$ preserves the inner product on $V$ induced by the Haar functional.
\end{enumerate} Then there is a unique $\ast$-morphism $\Psi : \mathcal{O}(U_q(N)) \rightarrow Q$ such that $(id \otimes \Psi)\rho_u=\rho$. 
\end{thm}

Before we go to the proof, we prove a lemma.

\begin{lem}\label{6.7}
In the notation of Theorem \ref{6.6}, let $\mu$ be the multiplication of $\mathcal{O}(S_q^{2N-1})$ restricted to $V \otimes V$. Then $\ker(\mu)=\mathrm{im}(\hat{R}-qI)$.
\end{lem}

\begin{proof}
We first claim that $\ker(\mu)=span\{z_i \otimes z_j -qz_j \otimes z_i; i < j\}$. Clearly, $z_i \otimes z_j -qz_j \otimes z_i \in \ker(\mu)$. Moreover, $v=\sum_{i,j}c_{ij}z_i \otimes z_j \in V \otimes V$ ($c_{ij} \in \mathbb{C}$) is in $\ker(\mu)$ if and only if \[0=\sum_{ij}c_{ij}z_iz_j=\sum_{i<j}(c_{ij}+q^{-1}c_{ji})z_iz_j+\sum_ic_{ii}z_i^2 \qquad (\text{by } \eqref{eq16}).\] It follows, by the linear independence of $\{z_iz_j, z_i^2; i<j\}$ (Proposition \ref{6.5}), that $c_{ii}=0$ for all $i$ and $c_{ij}+q^{-1}c_{ji}=0$ for all $i<j$. Hence, $v$ reduces to the form $v=\sum_{i<j}c_{ij}(z_i \otimes z_j - qz_j \otimes z_i)$, proving the claim.\\

On the other hand, using the definition of $\hat{R}$ given by  \eqref{eq10}, an easy computation gives
\begin{equation}\label{eq25}
(\hat{R}-qI)(z_i \otimes z_j)=
\begin{cases}
q^{-1}(z_i \otimes z_j - qz_j \otimes z_i) & i<j,\\
z_j \otimes z_i - qz_i \otimes z_j & i>j,\\
0 & i=j.
\end{cases}
\end{equation}
Hence $\mathrm{im}(\hat{R}-qI)=span\{z_i \otimes z_j - qz_j \otimes z_i; i < j\}=\ker(\mu)$.
\end{proof}

\begin{proof}[Proof of Theorem \ref{6.6}]
We start with the observation that since $V$ is a $Q$-comodule, we have that $\Delta(q^i_j)=\sum_k q^i_k \otimes q^k_j$ and $\epsilon(q^i_j)=\delta_{ij}$.\\

Now recall the map $\sigma$ from Remark \ref{4.4}. By Lemma \ref{6.2}, $\sigma$ is hermitian. Since $\sigma=t\hat{R}$ and $t$ is real, $\hat{R}$ is also hermitian. The multiplication of $\mathcal{O}(S_q^{2N-1})$ is a $Q$-comodule morphism, $\mathcal{O}(S_q^{2N-1})$ being a $Q$-comodule algebra. Since $V$ is assumed to be a subcomodule of $\mathcal{O}(S_q^{2N-1})$, the restriction $\mu$ of the multiplication to $V \otimes V$ is also a $Q$-comodule morphism. We recall that $V \otimes V$ is a $Q$-comodule with coaction $\rho_{V \otimes V}(v \otimes w)=v_0 \otimes w_0 \otimes v_1w_1$, where $\rho(v)=v_0 \otimes v_1$. As $\mu$ is a $Q$-comodule morphism, we have $\rho \mu=(\mu \otimes id)\rho_{V \otimes V}$. Thus $\ker(\mu)=\mathrm{im}(\hat{R}-qI)$ (from Lemma \ref{6.7} above) is also a $Q$-comodule whose orthogonal complement is the image of $\hat{R} + q^{-1}I$, by Remark \ref{3.3}. Hence, by Proposition \ref{4.12}, $\mathrm{im}(\hat{R} + q^{-1}I)$ is a $Q$-comodule. $\hat{R}$ has two eigenvalues, namely, $q$ and $-q^{-1}$, the corresponding eigenspaces being $\mathrm{im}(\hat{R}+q^{-1}I)$ and $\mathrm{im}(\hat{R}-qI)$, respectively. Since $\rho$ preserves both the eigenspaces, $\hat{R}$ becomes a $Q$-comodule morphism. By Proposition \ref{3.4}, $\textbf{q}$ then satisfies the FRT relations \eqref{eq6}, \eqref{eq7}, \eqref{eq8}.\\

By assumption, $\rho$ preserves the relation $\sum_{i=1}^{N}z_iz^*_i=1$. Applying $\rho$ to both sides, comparing coefficients and using Proposition \ref{6.5} we get that $\textbf{q}\textbf{q}^*=I_n$.\\

Thus, Lemma \ref{5.8} yields a unique $\ast$-morphism $\Psi : \mathcal{O}(U_q(N)) \to Q$ such that $\Psi(u^i_j)=q^i_j$. Moreover, $(id \otimes \Psi)\rho_u$ and $\rho$ agree on the generators $z_i$, hence they are equal. For any other $\Psi'$ satisfying $(id \otimes \Psi')\rho_u=\rho$, we get by evaluating both sides on the generators $z_i$, that $\Psi'(u^i_j)=q^i_j$. Hence, by the uniqueness in Lemma \ref{5.8}, $\Psi'=\Psi$. 
\end{proof}

\begin{rem}\label{6.8}
The equation \eqref{eq19} can also be written as $\sum_{i=1}^{N}q^{-2i}z_i^*z_i=q^{-2}$. Now applying $\rho$ to both sides, comparing coefficients and using Proposition \ref{6.5}, we see that $\textbf{q}$ satisfies $E\overline{\textbf{q}}E^{-1}\textbf{q}^t=\textbf{q}^tE\overline{\textbf{q}}E^{-1}=I_n$, where $E$ is the matrix \[\dfrac{1}{q^{n-1}[n]_q}\hspace{.09cm}\text{diag}\hspace{.09cm}(1,q^2,q^4,\cdots,q^{2(n-1)}),\qquad [n]_q:=\dfrac{q^n-q^{-n}}{q-q^{-1}}.\] See \cites{MR3275038, MR1382726}.
\end{rem}

We finally have the following theorem.

\begin{thm}\label{6.9}
Consider the category $\mathcal{C}$ consisting of Hopf $\ast$-algebras satisfying the hypotheses of Theorem \ref{6.6} as objects and Hopf $\ast$-algebra morphisms intertwining the coactions as morphisms. Then $\mathcal{O}(U_q(N))$ is a universal object in this category.
\end{thm}

\begin{proof}
By Proposition \ref{6.4} and Lemma \ref{6.3}, $\mathcal{O}(U_q(N))$ is an object in this category. Then Theorem \ref{6.6} shows that $\mathcal{O}(U_q(N))$ is universal with that property.
\end{proof}

\section{An application}

We provide an application of the main result of the previous section, namely, that of determining the quantum isometry group of the odd sphere, in the sense of \cite{MR2555012}. We begin by defining unitary representations of a compact quantum group.

\begin{defn}\label{7.1}\cite{MR901157}
A unitary representation of a compact quantum group $(S, \Delta)$ on a Hilbert space $H$ is a map $u$ from $H$ to $H \otimes S$ such that the element $\tilde{u} \in M(K(H) \otimes S)$ given by $\tilde{u}(\xi \otimes b)=u(\xi)(1 \otimes b)$ ($\xi \in H$, $b \in S$) is a unitary satisfying $(id \otimes \Delta)\tilde{u}=\tilde{u}_{(12)}\tilde{u}_{(13)}$, where for an operator $X \in B(H_1 \otimes H_2)$ we have denoted by $X_{(12)}$ and $X_{(13)}$ the operators $X \otimes I_{H_2} \in B(H_1 \otimes H_2 \otimes H_2)$ and $\Sigma_{23}X_{(12)}\Sigma_{23}$ respectively (where $\Sigma_{23}$ being the unitary on $H_1 \otimes H_2 \otimes H_2$ that switches the two copies of $H_2$).
\end{defn}

Irreducible unitary representations of the quantum group $SU_q(N)$ are indexed by Young tableaux $\lambda=(\lambda_1,\cdots,\lambda_N)$, where $\lambda_i$'s are nonnegative integers $\lambda_1 \geq \cdots \geq \lambda_N$ \cite{MR943923}*{page 42}. Let $H_{\lambda}$ be the carrier Hilbert space corresponding to $\lambda$ whose basis elements are parametrized by arrays of the form
\[\textbf{r}=\begin{bmatrix}
r_{11} & r_{12} & \cdots & r_{1,N-1} & r_{1N}\\
r_{21} & r_{22} & \cdots & r_{2,N-1}\\
\vdots & \vdots\\
r_{N-1,1} & r_{N-1,2}\\
r_{N1}\\
\end{bmatrix},\]
where $r_{ij}$'s are integers satisfying $r_{1j}=\lambda_j$ for $j=1,\cdots,N$, $r_{ij} \geq r_{i+1,j} \geq r_{i,j+1} \geq 0$ for all $i,j$. Such arrays are known as Gelfand-Tsetlin (GT) tableaux (see \cite{MR2458039}*{page 29} for details). For a GT tableaux $\textbf{r}$, $\textbf{r}_i$ will denote its $i$th row.\\

It was shown by Woronowicz (\cite{MR901157}) that any compact quantum group possesses a Haar functional in the sense of Definition \ref{4.8}. Let us denote the Haar functional of $SU_q(N)$ again by $\textbf{h}$. Let $L_2(SU_q(N))$ denote the corresponding G.N.S space and $L_2(S_q^{2N-1})$ denote the closure of $S_q^{2N-1}$ in $L_2(SU_q(N))$.

\begin{prop}\label{7.2} \cite{MR2458039}*{page 34}
Assume $N > 2$. The restriction of the right regular representation of $SU_q(N)$to $L_2(S_q^{2N-1})$ decomposes as a direct sum of the irreducibles, with each copy occurring exactly once, given by the Young tableau $\lambda_{n,k}=(n+k,k,k,\cdots,k,0)$ with $n,k \in \mathbb{N}_0.$
\end{prop}

Let the irreducible representation corresponding to the Young tableau $\lambda_{n,k}=(n+k,k,\cdots,k,0)$ be denoted by $V_{n,k}$. According to our previous notation $V=V_{1,0}$, the irreducible with Young tableau $\lambda_{1,0}=(1,0,\cdots,0)$. Moreover, observe that $V_{0,1}$, the irreducible with Young tableau $\lambda_{0,1}=(1,1,\cdots,1,0)$, is the conjugate corepresentation to $V$. In our previous notation, this is nothing but the span of $z_1^*, \cdots, z_N^*$, i.e., $V_{0,1}=V^*=\{v^* \mid v \in V\}$.\\

Consider the space \[W^{\otimes (n,k)}:=\underbrace{V_{1,0} \otimes \cdots \otimes V_{1,0}}_\text{n times} \otimes \underbrace{V_{0,1} \otimes \cdots \otimes V_{0,1}}_\text{k times}.\] Let us denote the image of it in the algebra under multiplication by \[W^{\bullet (n,k)}:=\underbrace{V_{1,0} \bullet \cdots \bullet V_{1,0}}_\text{n times} \bullet \underbrace{V_{0,1} \bullet \cdots \bullet V_{0,1}}_\text{k times}.\] Let $\Lambda^{\otimes (n,k)}$ be the set of the Young tableaux for the irreducible representations occurring in the decomposition of $W^{\otimes (n,k)}$ and $\Lambda^{\bullet (n,k)}$ be the corresponding set for the decomposition of $W^{\bullet (n,k)}$. The following proposition gives a description of $W^{\bullet (n,k)}$.

\begin{prop}\label{7.3}
\begin{enumerate}[i)]
\item $V_{n,k}$ occurs with multiplicity exactly one in the orthogonal decomposition of $W^{\bullet (n,k)}$ into irreducibles;
\item If a copy of $V_{m,l}$ occurs in the irreducible decomposition of the orthogonal complement $(V_{n,k})^{\perp}$ of $V_{n,k}$ in $W^{\bullet (n,k)}$ then we must have:
\[\text{either,} \quad l < k \text{ and }m \leq n+k-l;\]
\[\text{or,} \quad l=k \text{ and } m < n.\]
\end{enumerate}
\end{prop}

\begin{proof}
By Proposition \ref{7.2}, it follows that any $\lambda \in \Lambda^{\bullet (n,k)}$ must be of the form $\lambda_{m,l}=(m+l,l,\cdots,l,0)$ and has multiplicity one. It is known (see e.g., \cite{MR1358358}*{page 326, Proposition 10.1.16}) that any $\lambda \in \Lambda^{\otimes (n,k)}$ is dominated by $\lambda_{n,k}=(n+k,k,\cdots,k,0)$ which is equivalent to $l \leq k$ and $m+l \leq n+k$. ii) follows immediately from these remarks.\\

For i), let us first remark that $\mathcal{O}(SU_q(N))$, hence $\mathcal{O}(S_q^{2N-1})$ is an integral domain (see \cite{MR1614943}*{page 98}). Now, if $v$ and $\overline{v}$ are the primitive vectors for $V_{1,0}$ and $V_{0,1}$, respectively, then $v^n\overline{v}^k \in W^{\bullet (n,k)}$ is nonzero and belongs to the weight space corresponding to $\lambda_{n,k}$, which follows from the definition of weight spaces as in \cite{MR1358358}*{page 324} and the fact that each $K_i$ is group-like (\cite{MR1358358}*{page 281}, so that $K_i(xy)=(K_ix)(K_iy)).$
\end{proof}

The following will be useful in constructing a non-commutative structure on the sphere.

\begin{prop}\label{7.4} \cite{MR2458039}*{page 35}
Let $\Gamma_0$ be the set of all GT tableaux $\textbf{r}^{nk}$ given by 
\[r_{ij}^{nk}=
\begin{cases}
n+k & \text{if} \quad i=j=1\\
0 & \text{if} \quad i=1, j= N\\
k & \text{otherwise},
\end{cases}
\] for some $n,k \in \mathbb{N}$. Let $\Gamma_0^{nk}$ be the set of all GT tableaux with top row $(n+k,k,k,\cdots,k,0)$. Then the family of vectors \[\{e_{\textbf{r}^{nk},\textbf{s}} \mid n,k \in \mathbb{N}, \textbf{s} \in \Gamma^{nk}_0\}\] form a complete orthonormal basis for $L_2(S_q^{2N-1})$.
\end{prop}

We recall the following definition.

\begin{defn}\label{7.5}\citelist{\cite{MR1482228}*{page 93} \cite{MR1303779}}
A spectral triple $(A,H,D)$ of compact type is given by a unital $\ast$-algebra with a faithful representation $\pi : A \rightarrow B(H)$ on the Hilbert space $H$ together with a self-adjoint operator $D=D^*$ on $H$ with the following properties:
	
\begin{enumerate}[i)]
\item[i)] The resolvent $(D-\lambda)^{-1}$, $\lambda \not \in \mathbb{R}$, is a compact operator on $H$;
\item[ii)] $[D,a]:= D\pi(a)-\pi(a)D \in B(H)$, for any $a \in A$.
\end{enumerate}
\end{defn}

In what follows, by a spectral triple, we mean a spectral triple of compact type. Let us now put a non-commutative structure on the quantum sphere.

\begin{thm}\label{7.6}\cite{MR2458039}*{page 39}
Let $A=\mathcal{O}(S_q^{2N-1})$ and $H$ be $L_2(S_q^{2N-1})$. Take $\pi$ to be the inclusion. Finally, define the operator $D : e_{\textbf{\textbf{r},\textbf{s}}} \mapsto d(\textbf{r})e_{\textbf{\textbf{r},\textbf{s}}}$ on $L_2(S_q^{2N-1})$ where the $d(\textbf{r})$'s are given by \[d(\textbf{r}^{nk})=
\begin{cases}
-k & \text{if} \quad n=0\\
n+k & \text{if} \quad n > 0.
\end{cases}\] Then $(A,H,D)$, as constructed above, is a spectral triple on $\mathcal{O}(S_q^{2N-1})$.
\end{thm}

Now let us recall the notion of the quantum isometry group of a non-commutative manifold.

\begin{defn}\label{7.7} \cite{MR2555012}*{page 2538}
A quantum family of orientation preserving isometries for the spectral triple $(A,H,D)$ is given by a pair $(S,u)$ where $S$ is a unital $C^*$-algebra and $u$ is a linear map from $H$ to $H \otimes S$ such that $\tilde{u}$ given by $\tilde{u}(\xi \otimes b)=u(\xi)(1 \otimes b)$, extends to a unitary element of $M(\mathcal{K}(H) \otimes S)$ satisfying 
\begin{enumerate}[i)]
\item for every state $\phi$ on $S$, $u_{\phi}D=Du_{\phi}$ where $u_{\phi}=(id \otimes \phi) \circ \tilde{u}$;
\item $(id \otimes \phi) \circ {\rm ad}_u(a) \in A''$, for all $a \in A$ and for all state $\phi$ on $S$, where ${\rm ad}_u(x)=\tilde{u}(x \otimes 1)\tilde{u}^*$ for $x \in B(H)$.
\end{enumerate}

In case the $C^*$-algebra $S$ has a coproduct $\Delta$ such that $(S,\Delta)$ is a compact quantum group and $U$ is a unitary representation of $(S,\Delta)$ on $H$, it is said that $(S,\Delta)$ acts by orientation preserving isometries on the spectral triple.
\end{defn}

One considers the category $\textbf{Q}(A,H,D)$ whose objects are triples $(S,\Delta,u)$, where $(S,\Delta)$ is a compact quantum group acting by orientation preserving isometries on the given spectral triple, with $u$ being the corresponding unitary representation. The morphisms are homomorphisms of compact quantum groups that also ``intertwines" (see \cite{MR2555012}) the unitary representations. If a universal object $(\widetilde{S_0}, \Delta_0, u_0)$ (say) exists in this category then the $C^*$-algebra $S_0$ generated by $\{ (t_{\xi,\eta} \otimes {\rm id})({\rm ad}_{u_0}(a)),~\xi, \eta \in H,~q \in A\}$ is a compact quantum group and it is called the quantum group of orientation preserving isometries. Here, 
$t_{\xi,\eta}:  B(H) \rightarrow {\mathbb C}$ is given by $t_{\xi,\eta}(X)=<\xi, X\eta>$.\\

We record some results from \cite{MR2555012} for the reader's convenience.

\begin{thm}\label{7.8}\cite{MR2555012}*{page 2547}
Let $(A,H,D)$ be a spectral triple and assume that $D$ has a one dimensional eigenspace spanned by a unit vector $\xi$, which is cyclic and separating for the algebra $A$. Moreover, assume that each eigenvector of $D$ belongs to the dense subspace $A\xi$ of $H$. Then there is a universal object $(\widetilde{S_0},\Delta_0,u_0)$ in the category $\textbf{Q}(A,H,D)$.
\end{thm}

For the spectral triple in Theorem \ref{7.6}, the cyclic separating vector $\xi$ is $1_{\mathcal{O}(S_q^{2N-1})}$.

\begin{defn}\label{7.9}\cite{MR2555012}*{page 2548}
Let $(A,H,D)$ be the spectral triple in Theorem \ref{7.6}. Let $\widehat{\textbf{Q}}(A,H,D)$ be the category with objects $(S,\alpha)$ where $S$ is a compact quantum group with an action on $A$ such that
\begin{enumerate}[i)]
\item $\alpha$ is $\textbf{h}$ preserving;
\item $\alpha$ commutes with $\widehat{D}$, i.e., $\alpha \widehat{D}=(\widehat{D} \otimes id)\alpha$, where $\widehat{D}$ is the operator $A \rightarrow A$ given by $\widehat{D}(a)\xi=D(a\xi)$, $\xi$ as in Theorem \ref{7.8} which is $1_{\mathcal{O}(S_q^{2N-1})}$ in our case.
\end{enumerate}
\end{defn}

Note that the eigenspaces of $\widehat{D}$ and $D$ are in one-one correspondence. In fact, eigenspaces of $\widehat{D}$ are of the form $\{a \in A \mid a\xi \in V_{\lambda}\}$, where $V_{\lambda}$ is the finite dimensional eigenspace of $D$ with respect to eigenvalue $\lambda$.

\begin{prop}\label{7.10}\cite{MR2555012}*{page 2549}
There exists a universal object $S$ in the category $\widehat{\textbf{Q}}(A,H,D)$ and it is isomorphic to the $C^*$-subalgebra $S_0$ of $\widetilde{S_0}$ obtained in Theorem \ref{7.8}.
\end{prop}

We conclude by describing the quantum isometry group of the sphere. This generalizes \cite{MR2555012}*{Theorem 4.13, page 2559}.

\begin{lem}\label{7.11}
Given a compact quantum group $S$ with an action $\alpha$ on $A$, the following are equivalent:
\begin{enumerate}[i)]
\item $(S,\alpha)$ is an object of the category $\widehat{\textbf{Q}}(A,H,D)$, $(A,H,D)$ as in Theorem \ref{7.6};
\item $\alpha$ is linear (meaning it preserves $V$ as in Theorem \ref{6.6}) and preserves $\textbf{h}$;
\item $\alpha$ preserves each irreducible $V_{n,k}$, occurring in Proposition \ref{7.2}.
\end{enumerate}
\end{lem}

\begin{proof}
i) $\implies$ ii): Since $\alpha$ commutes with $\widehat{D}$, it preserves the eigenspaces of $\widehat{D}$, in particular $V$, and by definition of the category $\widehat{\textbf{Q}}(A,H,D)$, it preserves $\textbf{h}$.\\

ii) $\implies$ iii): Recall the notation $W^{\bullet (n,k)}$ from the discussion below Proposition \ref{7.2}. Clearly, $\alpha$ preserves $W^{\bullet (n,0)}$ as it is a homomorphism and preserves $V_{1,0}$. It also preserves $W^{\bullet (0,k)}$ because it is a $\ast$-homomorphism and preserves $V_{0,1}=V_{1,0}^*$. Hence $\alpha$ preserves $W^{\bullet (n,k)}$. We will use this fact to show $\alpha$ preserves $V_{n,k}$ for all $n$ and $k$. Let $\textbf{P}(k)$ be the statement ``$\alpha$ preserves $V_{n,k}$ for all $n$". We now proceed to prove this statement for all $k$ by induction. We break the proof in several steps.\\

$\textbf{Step 1}$: We prove that $\textbf{P}(0)$ holds. Thus we need to show $\alpha$ preserves $V_{n,0}$ for all $n$. We use induction on $n$. Clearly, $\alpha$ preserves $V_{0,0}$. Next, we assume that $\alpha$ preserves $V_{m,0}$ for all $m < n$. By Proposition \ref{7.3}, each irreducible contained in $(V_{n,0})^{\perp}$ is of the form $V_{m,0}$ with $m < n$, which is preserved by $\alpha$. Hence $\alpha$ preserves $V_{n,0}$, by Proposition \ref{4.12}.\\

$\textbf{Step 2}$: Now we assume $\textbf{P}(l)$ holds for all $l < k$, i.e., $\alpha$ preserves $V_{n,l}$ for all $n$ and for all $l < k$. We have to prove that $\textbf{P}(k)$ holds, i.e., $\alpha$ preserves $V_{n,k}$ for all $n$. We use induction on $n$ (with fixed $k$) to prove this.\\

$\textbf{Step 2a}$: We prove that $\alpha$ preserves $V_{0,k}$. By Proposition \ref{7.3}, each irreducible occurring  in $(V_{0,k})^{\perp}$ is of the form $V_{0,l}$ with $l < k$. By assumption $\alpha$ preserves $V_{0,l}$ for all $l < k$. Hence, by Proposition \ref{4.12}, $\alpha$ preserves $V_{0,k}$.\\

$\textbf{Step 2b}$: Next we assume that $\alpha$ preserves $V_{m,k}$ for all $m < n$ and prove that it also preserves $V_{n,k}$ ($k$ fixed). By Proposition \ref{7.3}, each irreducible contained in $(V_{n,k})^{\perp}$ is of the form $V_{m,l}$ with $l < k$ or $V_{m,k}$ with $m < n$. $\alpha$ preserves $V_{m,l}$ with $l < k$ (the main induction hypothesis that $\textbf{P}(l)$ holds for all $l < k$) and $V_{m,k}$ with $m < n$, by the induction hypothesis at the beginning of the present step. Hence, by Proposition \ref{4.12}, $\alpha$ preserves $V_{n,k}$. This completes the induction on $n$, proving $\alpha$ preserves $V_{n,k}$ for all $n$.\\

So we have proved $\textbf{P}(0)$ holds and $\textbf{P}(k)$ holds assuming $\textbf{P}(l)$ holds for all $l < k$. This completes the induction on $k$, thus completing the proof of ii) $\implies$ iii).\\

iii) $\implies$ i): Condition iii) implies that $\alpha$ leaves each eigenspaces of $\widehat{D}$ invariant, hence, it commutes with $\widehat{D}$. Since $\textbf{h}(1)=1$ and $\ker(\textbf{h})$ is the span of all $V_{n,k}$ with $n+k \neq 0$, $\alpha$ preserves $\ker(\textbf{h})$. But then, $a-\textbf{h}(a)1 \in \ker(\textbf{h})$, so that $\alpha(a-\textbf{h}(a)1)=\alpha(a)-\textbf{h}(a)(1 \otimes 1) \in \ker(\textbf{h}) \otimes S$, implying $(\textbf{h} \otimes id)(\alpha(a))=\textbf{h}(a)1$. Hence, $\alpha$ preserves $\textbf{h}$.
\end{proof}

\begin{thm}\label{7.12}
The quantum group of orientation preserving isometries for the spectral triple in Theorem \ref{7.6} is the compact quantum group $U_q(N)$.
\end{thm}  

\begin{proof}
By Proposition \ref{7.10}, there exists a universal object $(S,\alpha)$ in $\widehat{\textbf{Q}}(A,H,D)$. By ii) and iii) of Lemma \ref{7.11}, $\alpha$ preserves $\textbf{h}$ and leaves the algebra generated by $\{V_{n,k}; n,k \in \mathbb{N} \cup \{0\}\}$, i.e., $\mathcal{O}(S_q^{2N-1})$ invariant. Moreover, as each $V_{n,k}$ is finite dimensional, $\alpha(\mathcal{O}(S_q^{2N-1})) \subset \mathcal{O}(S_q^{2N-1}) \otimes \mathcal{S}$, where $\mathcal{S}$ is the dense Hopf $\ast$-algebra inside the compact quantum group $S$ mentioned after Definition \ref{4.9}. In particular, by ii) of Lemma \ref{7.11}, the coaction of $\mathcal{S}$ on $\mathcal{O}(S_q^{2N-1})$ satisfies the hypotheses of Theorem \ref{6.6}, hence is an object of the category $\mathcal{C}$ as in Theorem \ref{6.9}. Thus we have a unique morphism $\Psi : \mathcal{O}(U_q(N)) \to \mathcal{S}$ in $\mathcal{C}$. As $U_q(N)$ is the universal $C^*$-algebra corresponding to $\mathcal{O}(U_q(N))$, $\Psi$ extends to a $C^*$-algebra morphism (again denoted by $\Psi$) $U_q(N) \to S$. It is clearly a morphism of compact quantum groups and intertwines the two ($C^*$-algebraic) actions, hence a morphism in $\widehat{\textbf{Q}}(A,H,D)$.\\

On the other hand, it follows from Lemma \ref{7.11} that $U_q(N)$ is an object in the category $\widehat{\textbf{Q}}(A,H,D)$ as the canonical action of the compact quantum group $U_q(N)$ on $S_q^{2N-1}$ satisfies ii) of Lemma \ref{7.11}. Hence we have a unique morphism $\Theta : S \to U_q(N)$ in $\widehat{\textbf{Q}}(A,H,D)$. Clearly, $\Psi \Theta$ is the identity morphism, since $S$ is the universal object. To show that $\Theta \Psi$ is the identity morphism, we observe that it is a morphism of compact quantum groups, hence takes $\mathcal{O}(U_q(N))$ to itself. Since $\mathcal{O}(U_q(N))$ is universal in $\mathcal{C}$, $\Theta \Psi$ is identity, at least on $\mathcal{O}(U_q(N))$. We recall that $U_q(N)$ is the universal $C^*$-algebra generated by $\mathcal{O}(U_q(N))$, hence $\Theta \Psi$ lifts uniquely to $U_q(N)$, implying that $\Theta \Psi$ is also the identity morphism. Thus $S$ is isomorphic to $U_q(N)$.
\end{proof}

\section*{Acknowledgement}

The first author is grateful to Aritra Bhowmick, Jyotishman Bhowmick and Sugato Mukhopadhyay for helpful discussions that led to many improvements of the paper. The second author is partially supported by J.C. Bose National Fellowship and Research Grant awarded by D.S.T. (Govt. of India). Both the authors thank the anonymous referee for several suggestions and corrections leading to improvement of the paper.

\end{document}